\documentclass{article}
\usepackage{amssymb}

\usepackage{amsfonts}
\usepackage{amsmath}

\newtheorem{theorem}{Theorem}

\newtheorem{corollary}[theorem]{Corollary}

\newenvironment{proof}[1][Proof]{\noindent\textbf{#1.} }{\ \rule{0.5em}{0.5em}}
\input{tcilatex}

\begin{document}

\title{Bohr' inequality for large functions}
\author{Loai Shakaa and Yusuf Abu Muhanna \\
American University of Sharjah,\\
ymuhanna@aus.edu}
\maketitle

\begin{abstract}
We prove that the Bohr' radius for large functions is $e^{-\pi }.$
\end{abstract}

\section{\protect\underline{Introduction: }}

In this article, we shall investigate Bohr's phenomenon for spaces of large
functions. By large,

we mean analytic function that maps the disk into a hyperbolic domain ( it
misses at least two finite points). This includes almost all known classes
of analytic functions. Here we give a lower bound for Bohr's radius.

We say a class of analytic functions $\digamma $ satisfies the Bohr's
phenomenon,

if a function $f=\dsum\limits_{0}a_{n}z^{n}$ is in $\digamma $, the Bohr'
operator 
\begin{equation*}
M(f)=\dsum\limits_{0}\left\vert a_{n}z^{n}\right\vert
\end{equation*}

is uniformly bounded on some uniform disk $|z|\leq \rho ,$ with $\rho >0.$
[1], [3].

The largest such radius $\rho $ is called the Bohr's radius for the class $%
\digamma .$

Clearly: 1) $M(f+g)\leq M(f)+M(g)$

\ \ \ \ \ \ \ \ \ \ \ \ 2) $M(fg)\leq M(f)M(g).\qquad $

\qquad \qquad 3) $M(1)=1.$

This makes $\digamma $ into a Banach algebra, with norm $M(f)$. Let us
recall two

properties for Banach algebra:

1) Bohr's theorem: If $|f(z)|<1,$ for all $z\in U,$ then $M(f)<1$, when $%
|z|<1/3.$

2) Von Neumann inequality: $|p(f(z))|<||p||_{\infty ,}$ where $p(z)$ is a
polynomial.

Von Neumann showed the above inequality is true for the space of boubded
operators

on a Hilbert space, $L(H)$. [3], [5].

Dixon [5], showed that $the$ space 
\begin{equation*}
l_{\beta }^{1}=\left\{ x=(x_{1},x_{2},x_{3},\ldots ):\frac{1}{\beta }%
\sum_{j=1}^{\infty }|x_{j}|<\infty \right\} ,
\end{equation*}%
Dixon in [5], showed that the Von Neumann inequality is true, for $0<\beta
<1/3$ and not satisfied for any $\beta \geq 1/3.$

\bigskip

By a large function, we mean an analytic function on the unit disk whose
range misses at most two finite points from the complex plane. The following
theorem is center to all of this:

\begin{theorem}
\bigskip (Uniformization Theorem [4]), If $D$ is an open set missing at
least 2 points then there is a universal cover(conformal) from $U$ into $D.$
This cover is unique with the normalization $F(0)=a$ and $F^{\prime }(0)>0,$
for some $a\in D.$
\end{theorem}

\begin{corollary}
If $f(z)$ is analytic and maps $U$ into $D$ then there is a Schwartz
function $\varphi (z)$ so that
\end{corollary}

$f(z)=F(\varphi (z)).$

If we denote the universal covering of $D$ by $F,$ then $F$ defines a
hyperbolic metric on $D$ defined by%
\begin{equation}
\lambda (F(z))=\frac{1}{F^{\prime }(z)}\frac{1}{1-|z|^{2}}.
\end{equation}

\bigskip

In the following theorem of David Minda, [9], conformal means non-vanishing
derivative.

\begin{theorem}
Let $D$ be a hyperbolic domain, with hyperbolic metric $\lambda (w),$
\end{theorem}

and a conformal map $f:U\rightarrow D$ is onto. Then%
\begin{equation*}
\lambda (f(z))=\frac{1}{|f^{\prime }(z)|}\frac{1}{1-|z|^{2}}.
\end{equation*}

Proof: Divide $D$ into geodiscs. These geodiscs correpond to geodisks

in $U.$ Let $\mu $ be defined as%
\begin{equation*}
\mu (f(z)=\frac{1}{|f^{\prime }(z)|}\frac{1}{1-|z|^{2}}
\end{equation*}%
Then $\int_{\gamma }\mu (w)|dw|=\int_{c}\mu (f(z))f^{\prime }(z)|dz|=\int_{c}%
\frac{1}{1-|z|^{2}}|dz|,$ For all $c$ geodiscs in $U,$

Hence 
\begin{equation*}
\int_{\gamma }\mu (w)|dw|=\int_{\gamma }\lambda (w)|dw|,
\end{equation*}

for all geodisks $\gamma $ in $D.$ Hence they are the same.

\qquad Let us also recall the modular function: [11], [10], 
\begin{equation*}
J(z)=16z\dprod\limits_{1}^{\infty }\left[ \frac{1+z^{2n}}{1+z^{2n-1}}\right]
^{8},
\end{equation*}

$J(z)$ is $0$ only at $0$ and $J\neq 0,1,\infty $ 0n $|z|>0.$ Note that 
\begin{eqnarray}
-J(-z) &=&16z\dsum\limits_{0}^{\infty }A_{n}z^{n}, \\
A_{n} &>&0.  \notag
\end{eqnarray}

\bigskip This function is like a Koeba function for large function. We
immediately conclude that

\begin{equation}
\underset{|z|=r}{\max }|J(z)|=|J(-r)|.
\end{equation}

It is known that the coefficients $\{A_{n}\}$ are convex and increasing,
[10] and [11]. Hence,

by a theorem of Littlewood [8], we have

Lemma 1: if $\dsum\limits_{0}^{\infty }a_{k}$ is subordinate to $-J(-z)$
then $|a_{k}|\leq 16A_{k},$ for all $k.$

Also the following will be used:

Lemma II: $J(z)$ has radius of univalence $e^{-\pi /2},$ [10, p85].

In addition, as shown in the proof of Lemma II,

\begin{equation}
|J(-e^{-\pi })|=1,\text{ }|J(e^{-\pi })|=1/2.
\end{equation}

Consequently, by (1), (2), (3), we conclude that%
\begin{equation}
\underset{|z|\leq e^{-\pi }}{\max }|J(z)|=1.
\end{equation}

The function

\begin{equation*}
Q(z)=J(\exp (-\alpha \frac{1+z}{1-z})),\alpha <1
\end{equation*}

is a covering map into the domain $C\backslash \{0,1\}$ with $%
Q(0)=J(e^{-\alpha }).$

Other properties can be found in [10].

\section{This is our main theorem:}

\begin{theorem}
\bigskip Let $F(z)=\dsum\limits_{0}^{\infty }a_{n}z^{n}$ be analytic on $U$
and suppose that $F(U)$ misses at least

two points then 
\begin{equation*}
\dsum\limits_{1}^{\infty }\left\vert a_{n}z^{n}\right\vert \leq
dis(F(0),\partial F(U)),
\end{equation*}

for $|z|\leq e^{-\pi }=\allowbreak 4.\,\allowbreak 321\,4\times 10^{-2}.$
\end{theorem}

\begin{proof}
\bigskip\ By the Uniformization Theorem, There is a universal cover
(conformal) $G(z)$ onto $F(U),$ with

$G(0)=F(0)=a_{0}.$

Then $F(z)=G(\varphi (z)),$ where $|\varphi (z)|<1$ and $\varphi (0)=0.$
Suppose also that $F(U)$

misses points $a,b$, with 
\begin{equation*}
|a-a_{0}|=d(a_{0},\partial F(U)).
\end{equation*}%
Then the function 
\begin{equation*}
g(z)=\frac{F(z)-a}{b-a},
\end{equation*}%
misses $0,1$ and $g(z)=Q(\psi (z)),$ $\psi (0)=0.$ Then 
\begin{equation*}
g(0)=Q(\psi (0))=\frac{-a_{0}-a}{b-a}=J(e^{-\alpha }).
\end{equation*}

Next, let 
\begin{equation*}
h(z)=zg(z),
\end{equation*}

$h(z)$ is only $0$ at $0.$

If 
\begin{equation*}
F(z)=\dsum\limits_{0}a_{n}z^{n}.
\end{equation*}

Then 
\begin{equation*}
h(z)=\frac{a_{0}-a}{b-a}z+\frac{z}{b-a}\dsum\limits_{1}a_{n}z^{n}.
\end{equation*}

If 
\begin{equation*}
\delta =dis(0,\partial h(U)),
\end{equation*}%
then $\delta =\min |zg(z)|\leq \min |g(z)|=\min |\frac{F(z)-a}{b-a}|\leq |%
\frac{-a_{0}-a}{b-a}|=\left\vert \frac{dis(F(0),\partial F(U))}{b-a}%
\right\vert ,$

hence 
\begin{equation*}
\delta \leq \left\vert \frac{dis(F(0),\partial F(U))}{b-a}\right\vert .
\end{equation*}

On the other hand, since $h(z)$ is $0$ only at $0,$ $h/\delta =J(\omega ),$ $%
\omega (0)=0,$ [10]

and then $h(z)=\delta J(\omega ).$

Using the hyperbolic identity [4],%
\begin{equation*}
\lambda (w)d(w,\partial D)\leq 1,
\end{equation*}

By (2), we conclude that $\lambda (a_{0})=\frac{1}{|G^{\prime }(0)|}$ and by
(3), 
\begin{equation*}
dis(F(0),\partial F(U))\geq \delta |b-a|.
\end{equation*}

Then%
\begin{eqnarray*}
\lambda (a_{0})d(a_{0},\partial D) &\leq &1, \\
\frac{1}{|G^{\prime }(0)|}d(a_{0},\partial F(U)) &\leq &1 \\
\delta |b-a|/|G^{\prime }(0)| &\leq &1
\end{eqnarray*}
\end{proof}

\bigskip

\qquad \qquad and then

\begin{equation*}
\delta <\left\vert \frac{G^{\prime }(0)}{b-a}\right\vert .
\end{equation*}

\qquad \qquad \qquad As the coeff of $J(z)$ are convex increasing, by
Littlewood [8],

\begin{equation*}
\left\vert \frac{a_{n-1}}{b-a}\right\vert <\delta 16A_{n}<\left\vert \frac{%
dis(F(0),\partial F(U))}{b-a}\right\vert 16A_{n}
\end{equation*}

\begin{equation*}
\left\vert a_{n-1}\right\vert <|a|A_{n}=dis(F(0),\partial F(U))A_{n}
\end{equation*}

\begin{eqnarray*}
\dsum\limits_{1}\left\vert a_{n-1}\right\vert r^{n} &\leq &dis(F(0),\partial
F(U))\cdot 16\dsum\limits_{1}\left\vert A_{n}\right\vert r^{n} \\
&=&dis(F(0),\partial F(U))\cdot J(-r).
\end{eqnarray*}

\qquad \qquad (4) and (5) imply that $-J(-r)<1$ when $r\leq e^{-\pi }$This
proves Theorem 1.

\begin{corollary}
If $F(z)=\dsum\limits_{0}^{\infty }a_{n}z^{n}$ is analytic on $U$ and
suppose that $F(U)$ misses at least

two points, with $dis(F(0),\partial F(U))<1$ then $F(z)$ satisfies the von
Neumann

inequality for $|z|\leq e^{-\pi }=\allowbreak 4.\,\allowbreak 321\,4\times
10^{-2}$, in other words, for any polynomial $p(z)$ 
\begin{equation*}
p(F(z))\leq ||p||_{\infty }.
\end{equation*}
\end{corollary}

\section{ \protect\underline{Harmonic maps}: If $f=\protect\dsum%
\limits_{0}a_{n}z^{n}$ is analytic, define the operator 
\protect\begin{equation*}
M(f)=\protect\dsum\limits_{0}^{\infty }\left\vert a_{n}z^{n}\right\vert .
\protect\end{equation*}%
}

Clearly: 1) $M(f+g)\leq M(f)+M(g)$

\ \ \ \ \ \ \ \ \ \ \ \ 2) $M(fg)\leq M(f)M(g).$

\bigskip

\begin{theorem}
Let $f(z)=h(z)+\overline{g(z)}$ be harmonic with $h$ missing at least $2$
points,

\begin{theorem}
$h(0)=a_{0},$ $g(0)=0$ and $g^{\prime }(z)=\mu (z)h^{\prime }(z)$ then

\begin{equation*}
M(f)\leq (1+|\mu (z)|)d(a_{0},\partial (h(U))
\end{equation*}

for $|z|\leq e^{-\pi }=\allowbreak 4.\,\allowbreak 321\,4\times 10^{-2}.$
\end{theorem}
\end{theorem}

\begin{proof}
$M(g)=\dint\limits_{0}^{r}M(g^{\prime })dr\leq
\dint\limits_{0}^{r}M(a)M(h^{\prime })dr$

\begin{proof}
Let $|z|<1/3.$ Then 
\begin{equation*}
M(g)\leq \dint\limits_{0}^{r}M(h^{\prime })dr=M(h)-|a_{0}|=M(h-a_{0}).
\end{equation*}

Let $|z|<e^{-\pi },$ then%
\begin{eqnarray*}
M(g) &<&d(a_{0},\partial (h(U)), \\
M(f) &=&M(h)+M(g) \\
&\leq &(1+|\mu (z)|)d(a_{0},\partial (h(U))
\end{eqnarray*}
\end{proof}
\end{proof}

\end{document}